\documentclass[11pt, oneside]{amsart}   	
\usepackage{amsmath,amsthm,amscd,amsfonts, amssymb, mathrsfs}
\usepackage{geometry}                		
\usepackage{graphicx}				
\usepackage{amssymb}
\newcommand{\sect}[1]{\section{#1}\setcounter{equation}{0}}
\newcommand{\subsect}[1]{\subsection{#1}}

\font\mbn=msbm10 scaled \magstep1
\font\mbs=msbm7 scaled \magstep1
\font\mbss=msbm5 scaled \magstep1
\newfam\mbff
\textfont\mbff=\mbn
\scriptfont\mbff=\mbs
\scriptscriptfont\mbff=\mbss   

\newcommand{\N}       { \mathbb{N}}
  
\newcommand\Co           {{\mathbb C}}

\newtheorem{Th}{Theorem}[section]
\newtheorem{Lm}[Th]{Lemma}

\newtheorem{Proposition}[Th]{Proposition}
\newtheorem{R}[Th]{Remark}
\newtheorem*{Claim}{Claim}

\begin{document}
 
\title[On the Stable Rank of Stein Algebras]{On the Bass Stable Rank of Stein Algebras}
\author{Alexander Brudnyi}
\address{Department of Mathematics and Statistics\newline
\hspace*{1em} University of Calgary\newline
\hspace*{1em} Calgary, Alberta\newline
\hspace*{1em} T2N 1N4}
\email{abrudnyi@ucalgary.ca}

\thanks{Research supported in part by NSERC}

\begin{abstract}
We compute the Bass stable rank of the ring $\Gamma(X,\mathcal O_X)$ of global sections of the structure sheaf $\mathcal O_X$ on a finite-dimensional Stein space $(X,\mathcal O_X)$ and then apply this result to the problem of the factorization of invertible holomorphic matrices on $X$.
 \end{abstract}
 
\date{}

\maketitle

\sect{Formulation of Main Results}
\subsection{}
Let $A$ be an associative ring with identity $1$. An element $a=(a_1,\dots, a_n)\in A^n$ is  {\em unimodular} if there exists $b=(b_1,\dots, b_n)\in A^n$ such that $\langle b,a\rangle=\sum_{i=1}^n b_i a_i=1$. By $U_n(A)\subset A^n$ we denote the set of unimodular elements of $A^n$. An element $a=(a_1,\dots, a_n)\in U_n(A)$ is said to be {\em reducible} if there exist $c_1,\dots, c_{n-1}\in A$ such that
\[
(a_1+c_1 a_n,\dots, a_{n-1}+c_{n-1}a_n)\in U_{n-1}(A).
\]
$A$ is said to have a stable rank at most $n-1$ if every $a\in U_n(A)$ is reducible. The {\em stable rank} of $A$, denoted by ${\rm sr}(A)$, is the least $n-1$ with this property.

The concept of the stable rank introduced by Bass \cite{B}  plays an important role in some stabilization problems of algebraic $K$-theory  analogous to that of dimension in topology.
Despite a simple definition, ${\rm sr}(A)$ is often quite difficult to calculate even for relatively uncomplicated rings $A$ (cf. \cite{V1}). In this note we compute the stable rank of the ring $\Gamma(X,\mathcal O_X)$ of global sections of the structure sheaf $\mathcal O_X$ on a finite-dimensional Stein space $(X,\mathcal O_X)$. (For basic facts about  Stein spaces
we refer the readers to the book \cite{GR}.)
Our research is motivated by the recent work of
Ivarsson and Kutzschebauch \cite{IK1}, \cite{IK2}
which solves the Vaserstein problem on the factorization of invertible holomorphic matrices on a finite-dimensional reduced Stein space posed by Gromov \cite{G}, see section 1.2 below.

Recall that for a finite-dimensional complex analytic space $(X,\mathcal O_X)$ there is a natural algebra homomorphism $\hat{\,}: \Gamma(X,\mathcal O_X)\rightarrow C(X)$ with image $\mathcal O(X)$ the ring of holomorphic functions on $X$, injective if $(X,\mathcal O_X)$ is reduced. A space $(X,\mathcal O_X)$  is said to be {\em Stein} if it is {\em holomorphically convex} (i.e., for each infinite discrete set  $D\subset X$ there exists $ f\in \mathcal O(X)$ which is unbounded on $D$) and {\em holomorphic separable} (i.e., for all $x,y\in X$, $x\ne y$, there exists $ f\in \mathcal O(X)$  such that $ f(x)\ne  f(y)$). 
In what follows, sometimes for brevity we omit the structure sheaf $\mathcal O_X$ in the notation of the Stein space, i.e., write $X$ instead of $(X,\mathcal O_X)$. By ${\rm dim}\, X$ we denote the complex dimension of $X$.

Our main result reads as follows.
\begin{Th}\label{teo1}
Let $(X,\mathcal O_X)$ be a finite-dimensional Stein space.  Then \[
{\rm sr}(\Gamma(X,\mathcal O_X))={\rm sr}(\mathcal O(X))=\left\lfloor \frac 1 2\, {\rm dim}\, X \right\rfloor +1.
\]
\end{Th}

The particular case of the theorem for $X=\mathbb C$ was proved earlier in \cite{Ru} and \cite{CS1}. 

\subsection{}
In this part using Theorem \ref{teo1} we extend and sharpen some results of \cite{IK1}, \cite{IK2}.\smallskip

For an associative ring $A$ with $1$ by $M_{k,n}(A)$ we denote the set of $k\times n$ matrices with entries in $A$, by $GL_n(A)\subset M_{n,n}(A)$ the group of invertible matrices, and by $SL_n(A)\subset GL_n(A)$ the subgroup of matrices with determinant $1$. 

Recall that a matrix in $SL_n(A)$ is said to be {\em elementary} if it differs from the identity matrix $I_n$ by at most one non-diagonal entry.

By $E_n(A)$ we denote the subgroup of $SL_n(A)$ generated by all elementary matrices. 

Let $t_n(A)$ denote the minimal number $t$ such that every matrix in $E_n(A)$ is the product of $t$ matrices such that each of them is either upper triangular with $1$ along the main diagonal or lower triangular with $1$ along the diagonal. 

It is well-known that if $A$ is a field, then
$E_n(A)=SL_n(A)$ and each $t_n(A)<\infty$. In general this is not always true. For instance, for $A= \mathbb F [x_1,\dots, x_d]$,  the ring of polynomials in $d$ indeterminates over a field, $E_2(A)\subsetneq SL_2(A)$ if $d\ge 2$, see \cite[Prop.\,(7.3)]{C}, but
$E_n(A)=SL_n(A)$ for all $d$ and $n\ge 3$, see \cite[Cor.\,6.7]{S}.\footnote{Note that for $d=1$, $E_n(A)=SL_n(A)$ for all $n$ as $A$ is an Euclidean ring.} However, even in this case  $t_n(A)=\infty$ if $\mathbb F$ is of infinite transcendence degree over its prime field, see \cite[Prop.\,(1.5)]{W}. 

Further, it was proved by Vaserstein in \cite[Th.\,4]{V3} that for $A=C(X)$, the algebra of complex-valued continuous functions on a $d$-dimensional normal topological space $X$,  $E_n(A)$ coincides with the set of {\em null-homotopic} maps, i.e., maps in $SL_n(A):=C(X, SL_n(\mathbb C))$ homotopic (in this class of maps) to a constant map, and, moreover,  there exists a constant $v(d)\in\N$ depending on $d$ only such that $\sup_n t_n(A)\le v(d)$  (see also \cite[Lm.\,7]{DV}). 

A similar problem for the algebra $\mathcal O(X)$, where $X$ is a finite-dimensional reduced Stein space, was posed by Gromov in \cite[3.5.G]{G}  (the paper is devoted to the extension of the classical Oka-Grauert theorem) 
and solved recently by Ivarsson and Kutzschebauch in \cite{IK1} (see also \cite{IK2}) based on \cite[Th.\,4]{V3} and \cite[Th.\,8.3]{F}. 

Specifically, they proved that $E_n(\mathcal O(X))$ coincides with the set of null-homotopic holomorphic maps of $SL_n(\mathcal O(X)):=\mathcal O(X, SL_n(\mathbb C))$ and that the number $t_n(\mathcal O(X))$ is bounded from above by a constant depending on $n$ and $d:={\rm dim}\, X$ tending to $\infty$ as $n\rightarrow\infty$. 
\begin{R}
{\rm Note that \cite[Th.\,3.1]{IK2} implies that $t_2(\mathcal O(X))\le v(2d)+2$  with $v(\cdot )$ as in \cite[Th.\,4]{V3} above.
This and \cite[Lm.\,7]{DV} produce a much better upper bound $\sup_n t_n(\mathcal O(X))\le v(2d)+3$.  

In addition, \cite[Lm.\,7]{DV},  Theorems 5.1, 5.2 and Proposition 4.1 of \cite{IK2} imply that $t_n(\mathcal O(X))=4$ for all $n\ge 2$ if $d=1$ and $\sup_n t_n(\mathcal O(X))=5$  if $d=2$. (In \cite{IK2} similar statements are formulated  for $t_2(\mathcal O(X))$ only.)}
\end{R}

Now, let us consider the case of the algebra $\Gamma(X,\mathcal O_X)$ for a finite-dimensional Stein space $X$. Clearly, the algebra homomorphism $\hat{\,}:\Gamma(X,\mathcal O_X)\rightarrow \mathcal O(X)$ induces a group homomorphism $\widehat{\,}_n: SL_n(\Gamma(X,\mathcal O_X))\rightarrow SL_n(\mathcal O(X))$, $(f_{ij})\mapsto (\hat{f}_{ij})$.

Using the main result of \cite{IK1},  some results of \cite{DV} and Theorem \ref{teo1} we prove
\begin{Proposition}\label{prop2.1}
\begin{itemize}
\item[(1)]
$H\in E_n(\Gamma(X,\mathcal O_X))$ if and only if $\widehat{H}_n\in E_n(\mathcal O(X))$.
Moreover, 
\[
\sup_n\, t_n(\Gamma(X,\mathcal O_X))\le v(2d)+5,
\]
where $d:={\rm dim}\, X$.
\item[(2)] There is a number $n(d)\in\N$ such that for all $n\ge n(d)$
\[
t_n(\Gamma(X,\mathcal O_X))\le 6.
\]
\item[(3)] If $d\in\{1,2\}$, then $SL_n(\Gamma(X,\mathcal O_X))=E_n(\Gamma(X,\mathcal O_X))$ for all $n$ and 
\[
t_n(\Gamma(X,\mathcal O_X))=4\ {for\ all}\ n\ge 2\quad {\rm if}\quad d=1\quad  {\rm and}\quad \sup_n\, t_n(\Gamma(X,\mathcal O_X))\le 7\quad {\rm if}\quad d=2.
\]
\end{itemize}
\end{Proposition}
\sect{Auxiliary Results} In this section we collect some results used in the proof of Theorem \ref{teo1}. 
\subsection{}
 Let $(X,\mathcal O_X)$ be a finite-dimensional Stein  space. By the Cartan and Oka theorem, the nilradical $\mathfrak{n}(\mathcal O_X)$ of $\mathcal O_X$ (i.e., the union of nilradicals of stalks $\mathcal O_x$, $x\in X$)
 is a coherent sheaf of ideals on $X$ and so by Cartan's Theorem B  we have the following exact sequence of global sections of sheaves
 \begin{equation}\label{eq2.1}
0\rightarrow \Gamma(X,\mathfrak{n}(\mathcal O_X))\rightarrow \Gamma(X,\mathcal O_X)\stackrel{r^*}{\rightarrow}\Gamma(X,\mathcal O_{{\rm red}\, X})\rightarrow 0,
\end{equation}
where 
$\mathcal O_{{\rm red}\, X}:=\mathcal O_X/\mathfrak{n}(\mathcal O_X)$ is the structure sheaf on the reduction of $X$.
\begin{Lm}\label{lem2.1}
 $\Gamma(X,\mathfrak{n}(\mathcal O_X))$ is the Jacobson radical of  $\Gamma(X,\mathcal O_X)$.
\end{Lm}
\begin{proof}
By definition, the {\em Jacobson radical }$\mathscr J(A)$ of a commutative ring $A$ is the intersection of all maximal ideals of $A$ and the Jacobson radical of $A/\mathscr J(A)$ is trivial. In our case, if $f\in \mathscr J(\Gamma(X,\mathcal O_{{\rm red}\, X}))$, then $\hat f(x)=0$ for all $x\in X$, i.e., $\hat f=0\in \mathcal O(X)$. Since $\hat{\,} : \Gamma(X, \mathcal O_{{\rm red}\, X})\rightarrow \mathcal O(X)$ is an isomorphism of algebras, $f=0$.
Thus $\mathscr J(\Gamma(X,\mathcal O_{{\rm red}\, X}))=\{0\}$  and 
due to \eqref{eq2.1}, $\mathscr J(\Gamma(X,\mathcal O_X))\subset\Gamma(X,\mathfrak{n}(\mathcal O_X))$. 

Next, let $g\in \Gamma(X,\mathfrak{n}(\mathcal O_X))$. In order to check that $g$ is in the Jacobson radical one must prove that $1-fg$ is invertible for all $f\in\Gamma(X,\mathcal O_X)$, see, e.g., \cite[Prop.\,1.9]{AM}. Indeed, by the definition of the nilradical, there exist a locally finite open cover $(U_i)_{i\in I}$ of $X$ and a family $\{n_i\}_{i\in I}\subset\N$ such that
$g^{n_i}|_{U_i}=0$, $i\in I$. This implies that 
$h:=\sum_{i=0}^\infty (fg)^i$ is a well-defined section of  $ \Gamma(X,\mathcal O_X)$. Moreover, for each $i\in I$,
\[
(1-fg)|_{U_i}\cdot h|_{U_i}=(1-fg)|_{U_i}
\cdot \sum_{k=0}^{n_i-1}(fg)^k|_{U_i}=(1-(fg)^{n_i})|_{U_i}=1.
\]
 This shows that $(1-fg)\cdot h=1$, as required. Thus $g\in \mathscr J(\Gamma(X,\mathcal O_X))$ and so $\Gamma(X,\mathfrak{n}(\mathcal O_X))\subset \mathscr J(\Gamma(X,\mathcal O_X))$.
 
This completes the proof of the lemma.
\end{proof}
\begin{R}\label{rem1}
{\rm (1) Arguing as in the proof of the lemma one obtains that if $s(t)=\sum_{i=0}^\infty c_i t^i$ is a formal power series in $t$, then $s(g)$ is a well-defined section of $\Gamma(X,\mathcal O_X)$ for each $g\in\Gamma(X,\mathfrak n(\mathcal O_X))$.\smallskip

\noindent (2) Due to Lemma \ref{lem2.1}, \cite[Lm.\,3]{V1} and the fact that $\mathcal O(X)$ is isomorphic to $\Gamma(X, \mathcal O_{{\rm red} X})$,
 \begin{equation}\label{equ2.2}
 {\rm sr}(\mathcal O(X))={\rm sr}(\Gamma(X, \mathcal O_{{\rm red} X}))=  {\rm sr}(\Gamma(X,\mathcal O_X)).
 \end{equation}
Moreover, if $X_i\subset X$, $i\in\N$, are connected components of $X$, then each $(X_i,\mathcal O_{X_i})$, $\mathcal O_{X_i}:=\mathcal O_X|_{X_i}$, is Stein and 
$\Gamma(X,\mathcal O_X)$ is isomorphic to the direct product of the  algebras $\Gamma(X_i,\mathcal O_{X_i})$, $i\in\N$. In particular,  cf. \cite[Lm.\,2]{V1},
\begin{equation}\label{equ2.3}
{\rm sr}(\Gamma(X,\mathcal O_X))=\sup_{i\in\N}\,{\rm sr}(\Gamma(X,\mathcal O_{X_i})).
\end{equation}
In view of \eqref{equ2.2}, \eqref{equ2.3}, without loss of generality we may assume that $(X,\mathcal O_X)$ in the hypothesis of Theorem \ref{teo1} is a connected reduced Stein space.}
\end{R}
\subsection{} We retain the notation of the previous section. For a  commutative unital algebra $A$ by $A^{-1}\subset A$ we denote the multiplicative group of invertible elements.
\begin{Lm}\label{lem2.2}
\[
(r^*)^{-1}\bigl(\Gamma(X,\mathcal O_{{\rm red}\, X})^{-1}\bigr)=\Gamma(X,\mathcal O_X)^{-1}.
\]
\end{Lm}
\begin{proof}
Let $f, f'\in \Gamma(X,\mathcal O_X)$ be such that $r^*(f):=g\in \Gamma(X,\mathcal O_{{\rm red}\, X})^{-1}$ and
$r^*(f')=g^{-1}$, cf. \eqref{eq2.1}. Since $r^*$ is an algebra homomorphism,   $r^*(1-ff')=0$, i.e., $u:=1-ff'\in \Gamma(X,\mathfrak n(\mathcal O_X))$. As in the proof of Lemma \ref{lem2.1}, $v:=\sum_{i=0}^\infty u^i$ is a well-defined section of $\Gamma (X,\mathfrak n(\mathcal O_X))$ inverse to $1-u$. Hence, $f\cdot (f'v)=1$, i.e., 
$f\in \Gamma(X,\mathcal O_X)^{-1}$.
\end{proof}
\begin{Lm}\label{lem2.3}
Suppose  $A\in GL_m(\Gamma(X,\mathcal O_X))$ is such that all entries of $I_m-A$ belong to $\Gamma(X,\mathfrak n(\mathcal O_X))$. Then $A$ is a product of $(m+4)(m-1)$ elementary matrices and a matrix $\exp(h)\cdot I_m$ for some $h\in \Gamma(X,\mathfrak n(\mathcal O_X))$.
\end{Lm}
\begin{proof}
Clearly, ${\rm det}\,A=1-g$ for some $g\in  \Gamma(X,\mathfrak n(\mathcal O_X))$. Thus  $\log(1-g):=\sum_{i=1}^\infty \frac{g^i}{i}\in \Gamma(X,\mathfrak n(\mathcal O_X))$, cf. Remark \ref{rem1}. Similarly, $\exp\bigl(\log(1-g)\bigr):=\sum_{i=0}^\infty \frac{(\log(1-g))^i}{i!}\in\Gamma(X,\mathcal O_X)$ and it is readily seen that $\exp\bigl(\log(1-g)\bigr)=1-g$. We set 
\[
h:=\frac{\log(1-g)}{m}.
\]
Then $A=\exp(h)\cdot I_m\cdot \tilde A$, where $\tilde A$ satisfies the hypotheses of the lemma and  ${\rm det}\, \tilde A=1$. Applying to $\tilde A$ the Gauss-Jordan elimination process using only addition operations we present $\tilde A$ as a product of $m(m-1)$ elementary matrices and a diagonal matrix $D$, ${\rm det}\,D=1$, satisfying the hypotheses of the lemma.  Due to the Whitehead lemma, $D$ is a product of $4(m-1)$ elementary matrices. 

This proves the required statement.
\end{proof}

\subsection{} Let $(X,\mathcal O_X)$ be a reduced Stein space of complex dimension $n$ and $F\in M_{m,l}(\mathcal O(X))$ be a holomorphic  $m\times l$ matrix with  $1\le l<m$.
\begin{Lm}\label{lem2.4}
Suppose the family of all minors of order $l$ of $F$ does not have common zeros. If $m-l\ge \frac{n}{2}$, then $F$ can be extended to a matrix in  $GL_m(\mathcal O(X))$. 
\end{Lm}
\begin{proof}
The matrix $F$ determines a trivial holomorphic subbundle $\xi$ of rank $l$ of the trivial holomorphic vector bundle $\theta^m:=X\times\mathbb C^m$ on $X$ (so that the columns of $F$ are holomorphic sections of $\xi$ linear independent at each point of $X$). Let $\theta^m/\xi$ be the holomorphic quotient bundle. Since $X$ is Stein, by Cartan's Theorem B there exists a holomorphic subbundle $\eta$ of $\theta^m$ isomorphic to $\theta^m/\xi$ such that $\xi\oplus \eta=\theta^m$. Thus, $\eta$ is {\em stably trivial}. 
Since, due to \cite{H}, $X$ is homotopic to a $CW$ complex of dimension $n$ and ${\rm rank}\,\eta= m-l\ge \frac{n}{2}$,  bundle $\eta$ is (topologically) trivial, see, e.g., \cite[Th.\,3.4.7, 9.1.5]{Hus}. Therefore, by the Grauert theorem, $\eta$ is holomorphically trivial. In global coordinates of $\theta^m$ holomorphic sections $s_1,\dots, s_{m-l}$ trivializing  $\eta$ determine an $(m-l)\times m$ matrix $F'$ such that $(F,F')\in GL_m(\mathcal O(X))$ is an invertible holomorphic matrix extending $F$. 
\end{proof}
\subsect{} Let $A$ be a  commutative ring with identity $1$ and $J(a)\subset A$ be the principal ideal generated by $a\in A$. 
\begin{Lm}\label{lem2.6}
An element of the form $(a_1,\dots, a_n,a)\in U_{n+1}(A)$, $n\in\N$, is reducible if and only if the map $U_n(A)\rightarrow U_n(A/J(a))$ induced by the quotient homomorphism is surjective. 
\end{Lm}
\begin{proof}
The proof is straightforward, see, e.g.,  the proof of the Proposition in \cite[Sect.\,3]{CL}. 
\end{proof}

\sect{Proof of Theorem \ref{teo1}}
\begin{proof}
As it was mentioned in Remark \ref{rem1}\,(2), it suffices to prove the theorem for a connected finite-dimensional reduced Stein space $(X,\mathcal O_X)$.  (In this case, $\Gamma(X,\mathcal O_X)\cong \mathcal O(X)$.) Under this assumption we set
\begin{equation}\label{eq3.1}
s(X):=\left\lfloor \frac 1 2\, {\rm dim}\, X \right\rfloor +1.
\end{equation}

First, we prove that
\begin{equation}\label{equ3.2}
{\rm sr}(\mathcal O(X))\le s(X).
\end{equation}

To this end we need to check  the following statement.
\begin{Claim}
Each element $(f_1,\dots,f_{s(X)},f)\in U_{s(X)+1} (\mathcal O_X)$ is reducible. 
\end{Claim}
This is obvious for $f= 0$ or $f\in \mathcal O(X)^{-1}$.  Next, assuming that $f\not\in \mathcal O(X)^{-1}\cup\{0\}$ by $\mathcal Z(f)\subset X$ we denote its zero locus. 
Since $X$ is connected, there exists an irreducible component $\tilde X$ of $X$ such that the complex analytic subset  $\mathcal Z(f)\cap\tilde X$ of $\tilde X$ has dimension ${\rm dim}\,\tilde X-1\, (\le {\rm dim}\, X-1)$.  By $X_f$ we denote the union of all such components $\tilde X$ and all irreducible components $\bar X$ of $X$ such that $\mathcal Z(f)\cap\bar X=\emptyset$. Let $X':={\rm cl}(X\setminus X_f)$. Then $(X_f,\mathcal O_{X_f})$, $\mathcal O_{X_f}:=\mathcal O_X|_{X_f}$, and $(X',\mathcal O_{X'})$, $\mathcal O_{X'}:=\mathcal O_X|_{X'}$, are complex analytic subspaces of $X$ and $f|_{X'}= 0$.  
\begin{Lm}\label{lem1}
If an element $(f_1,\dots, f_{s(X)},f)\in U_{s(X)+1}(\mathcal O(X))$ is reducible over $X_f$, i.e., there exist $h_1,\dots, h_{s(X)}\in \mathcal O(X_f)$ such that
\begin{equation}\label{eq3.3}
(f_1|_{X_f}-h_1 f|_{X_f},\dots, f_{s(X)}|_{X_f}-h_{s(X)} f|_{X_f})\in U_{s(X)}(\mathcal O(X_f)),
\end{equation}
then it is reducible.
\end{Lm}
\begin{proof}
Due to Cartan's Theorem B for $X$ there are some $h_i'\in\mathcal O(X)$ such that $h_i'|_{X_f}=h_i$, $i=1,\dots, s(X)$. From here, \eqref{eq3.3} and the fact that $f|_{X'}= 0$ we obtain that the family of holomorphic functions
$\{f_i-h_i' f\}_{1\le i\le s(X)}$ does not have common zeros. Therefore by the corona theorem for Stein spaces  
\[
(f_1-h_1' f,\dots, f_{s(X)}-h_{s(X)}' f)\in U_{s(X)}(\mathcal O(X)).
\]
Thus $(f_1,\dots, f_{s(X)}, f)\in U_{s(X)+1}(\mathcal O(X))$ is reducible. 
\end{proof}
Let $J(f)\subset\mathcal O(X_f)$ be the principal ideal generated by $f|_{X_f}$. In order to confirm the Claim, due to Lemmas \ref{lem2.6} and \ref{lem1}, it suffices to prove    the following result.
\begin{Proposition}\label{prop1.3}
The map $U_{s(X)}(\mathcal O(X_f))\rightarrow U_{s(X)}(\mathcal O(X_f)/J(f))$ induced by the quotient homomorphism is surjective. 
\end{Proposition}
\begin{proof}
Let ${\mathcal Z}'(f):=\mathcal Z(f)\cap X_f$ and
 $\mathcal J_f\subset \mathcal O_{X_f}$ be the sheaf of principal ideals generated by germs of $f$. Consider the complex analytic space $({\mathcal Z}'(f), (\mathcal O_{X_f}/\mathcal J_f)|_{{\mathcal Z}'(f)})$. Its reduction  is the complex analytic subspace $({\mathcal Z}'(f), \mathcal O_{{\mathcal Z}'(f)})$,
 $\mathcal O_{{\mathcal Z}'(f)}:=\mathcal O_{X_f}|_{\mathcal Z '(f)}$, of $X_f$. By $r$ we denote the corresponding reduction homomorphism of the structure sheaves. Applying Cartan's Theorem B  to the long exact cohomology sequence obtained from the short sequence of sheaves
 \[
 0\rightarrow \mathcal J_f\rightarrow\mathcal O_{X_f}\rightarrow \mathcal O_{X_f}/\mathcal J_f\rightarrow 0,
 \]
we get the following sequence
\begin{equation}\label{eq3.3'}
0\rightarrow \Gamma(X_f,\mathcal J_f)\rightarrow \Gamma(X_f,\mathcal O_{X_f})\rightarrow
\Gamma(X_f, \mathcal O_{X_f}/\mathcal J_f)\rightarrow 0.
\end{equation}
Since $\mathcal Z'(f)$ is of the complex codimension one in each irreducible component of $X_f$ having a nonvoid intersection with $\mathcal Z'(f)$, under the natural identification of $\Gamma(X_f,\mathcal O_{X_f})$ with $\mathcal O(X_f)$  the space $\Gamma(X_f,\mathcal J_f)$ coincides with $J(f)$. Similarly, since each section in $\Gamma(X_f, \mathcal O_{X_f}/\mathcal J_f)$ equals zero outside $\mathcal Z'(f)$, the latter space is naturally identified with $\Gamma(\mathcal Z'(f), \mathcal O_{X_f}/\mathcal J_f)$.  These and \eqref{eq3.3'} give the following exact sequence
\begin{equation}\label{eq3.4}
0\rightarrow J(f)\rightarrow \mathcal O(X_f)\rightarrow \Gamma(\mathcal Z'(f), \mathcal O_{X_f}/\mathcal J_f)\rightarrow 0.
\end{equation}
Thus the statement of the proposition is equivalent to the following one:\smallskip

\noindent (*) {\em The map $\phi: U_{s(X)}(\mathcal O(X_f))\rightarrow U_{s(X)}(\Gamma(\mathcal Z'(f), \mathcal O_{X_f}/\mathcal J_f))$ induced by the quotient homomorphism and the restriction to $\mathcal Z'(f)$ is surjective.}\smallskip

In turn, the reduction $r$ induces a homomorphism 
\[
r^*: \Gamma(\mathcal Z'(f), \mathcal O_{X_f}/\mathcal J_f))\rightarrow \mathcal O(\mathcal Z'(f)),
\]
surjective due to Cartan's Theorem B. 

In the sequel,  for a finite-dimensional complex vector space $V$ by $\mathfrak r^*$  we denote the linear map ${\rm id}_V\otimes r^*: V\otimes_{\mathbb C} \Gamma(\mathcal Z'(f), \mathcal O_{X_f}/\mathcal J_f))\rightarrow V\otimes_{\mathbb C}  \mathcal O(\mathcal Z'(f))$. (The choice of $V$ will be understood from the context.) In particular, $\mathfrak r^*={\rm id}_{\mathbb C^{s(X)}}\otimes r^*$ sends
$U_{s(X)}(\Gamma(\mathcal Z'(f), \mathcal O_{X_f}/\mathcal J_f))$ to $U_{s(X)}(\mathcal O(\mathcal Z'(f)))$.
Thus we obtain the following diagram
\begin{equation}\label{eq3.5}
U_{s(X)}(\mathcal O(X_f))\stackrel{\phi}{\longrightarrow} U_{s(X)}(\Gamma(\mathcal Z'(f), \mathcal O_{X_f}/\mathcal J_f))\stackrel{\frak r^*}{\longrightarrow}U_{s(X)}( \mathcal O(\mathcal Z'(f))),
\end{equation}
where $\mathfrak r^*\circ\phi$ is the map induced by the restriction to $\mathcal Z'(f)$.

To check (*), first, we prove
\begin{Lm}\label{lem3.4}
The map $\frak r^*\circ \phi:U_{s(X)}(\mathcal O(X_f))\rightarrow U_{s(X)}(\mathcal O(\mathcal Z'(f)))$ is surjective.
\end{Lm}
\begin{proof}
Due to the corona theorem on a Stein space,
sets $U_{s(X)}(\mathcal O(X_f))$ and $U_{s(X)}(\mathcal O(\mathcal Z'(f)))$ coincide with $\mathcal O(X_f,(\mathbb C^{s(X)})^*)$  and $\mathcal O(\mathcal Z'(f),(\mathbb C^{s(X)})^*)$, respectively; here $(\mathbb C^k)^*:=\mathbb C^k\setminus\{0\}$. In turn, $(\mathbb C^{s(X)})^*$ is homotopic to $\mathbb S^{2s(X)-1}$ 
(the $(2s(X)-1)$- dimensional unit Euclidean sphere) while,  due to \cite{H}, $\mathcal Z'(f)$ is homotopic to a $CW$ complex of dimension ${\rm dim}\, \mathcal Z'_f\, (={\rm dim}\, X_f -1\le {\rm dim}\, X -1<2s(X)-1)$, see \eqref{eq3.1}. 
Let us show that 

{\em Each map in $C(\mathcal Z'(f), (\mathbb C^{s(X)})^*)$ is extendable to a map in} $C(X_f,(\mathbb C^{s(X)})^*)$. 

Indeed, since the space $(\mathbb C^{s(X)})^*$ is a simple absolute neighbourhood retract, in order to prove the previous statement it suffices to check that $H^{l+1}(X,\mathcal Z'(f);\pi_l)=0$ for each $0<l<{\rm dim}\, (X_f\setminus\mathcal Z'(f))$, where $\pi_l$ is the $l$th homotopy group of $\mathbb S^{2s(X)-1}$, see, e.g., \cite[p.\,348,\,(5.3)]{Hu}.  Clearly, this is true for $0<l\le 2s(X)-2$ (for $s(X)>1$) because $\pi_l=0$ in this case. Next, if $l=2s(X)-1+m$ for some $m\ge 0$, then from the long exact cohomological sequence of the pair $(X_f,\mathcal Z'(f))$ with coefficients in $\pi_{2s(X)-1+m}$ we obtain 
\[
\begin{array}{r}
\displaystyle
\cdots\rightarrow H^{2s(X)-1+m}(\mathcal Z'(f);\pi_{2s(X)-1+m})\rightarrow H^{2s(X)+m}(X_f,\mathcal Z'(f);\pi_{2s(X)-1+m})\\
\\
\displaystyle \rightarrow H^{2s(X)+m}(X_f;\pi_{2s(X)-1+m})\rightarrow\cdots .
\end{array}
 \]
 The first term here is zero because $\mathcal Z'(f)$ is homotopic to a $CW$ complex of dimension $<2s(X)-1$ while the last term is zero because,  due to \cite{H},  $X_f$ is homotopic to a $CW$ complex of dimension ${\rm dim}\, X_f\le{\rm dim}\, X\le 2s(X)-1$.  Therefore the intermediate term of the above sequence is zero as well. This proves the required statement.
 
 In particular, each map in $\mathcal O(\mathcal Z'(f),(\mathbb C^{s(X)})^*)$ extends  to a map in $C(X_f, (\mathbb C^{s(X)})^*)$. Hence, by the Ramspott theorem \cite{R} it extends also to a map in $\mathcal O(X_f, (\mathbb C^{s(X)})^*)$. This shows that $\mathfrak r^*\circ\phi$ is surjective and completes the proof of the lemma.
\end{proof}

Now using the lemma let us prove the proposition. 

Let $\mathcal R:=\Gamma(\mathcal Z'(f),\frak n(\mathcal O_{X_f}/\mathcal J_f))$ denote the Jacobson radical of $\Gamma(\mathcal Z'(f),\mathcal O_{X_f}/\mathcal J_f)$, see Lemma \ref{lem2.1}. Due to Lemma \ref{lem3.4}, given $h=(h_1,\dots, h_{s(X)})\in U_{s(X)}(\Gamma(\mathcal Z'(f),\mathcal O_{X_f}/\mathcal J_f)$ there exists $f=(f_1,\dots, f_{s(X)})\in U_{s(X)}(\mathcal O(X_f))$ such that
\[
\mathfrak r^*(h)=f|_{\mathcal Z'(f)}\, (:=\mathfrak r^*(\phi(f))).
\]
In order to prove (*) we must show that there is some $\tilde h\in U_{s(X)}(\mathcal O(X_f))$ such that $\phi(\tilde h)=h$.

Observe that by Lemma \ref{lem2.4}, $\mathfrak r^*(h)$ considered as a column-matrix can be extended to a matrix $(\mathfrak r^*(h), F)\in GL_{s(X)}(\mathcal O(\mathcal Z'(f)))$. Since, due to Cartan's Theorem B, the homomorphism $r^*:\Gamma(\mathcal Z'(f),\mathcal O_{X_f}/\mathcal J_f)\rightarrow \Gamma(\mathcal Z'(f), \mathcal O_{\mathcal Z'(f)})\cong \mathcal O(\mathcal Z'(f))$ induced by the reduction map  is surjective, there is a matrix $\tilde F\in M_{s(X)\times (s(X)-1)}(\Gamma(\mathcal Z'(f),\mathcal O_{X_f}/\mathcal J_f))$  such that $\mathfrak r^*(\tilde F)=F$ (here $\mathfrak r^*={\rm id}_V\otimes r^{*}$ with $V:=M_{s(X)\times (s(X)-1)}(\mathbb C)$). Note that
\[
r^*({\rm det}(h, \tilde F))={\rm det}(\mathfrak r^*(h), F)\in \mathcal O(\mathcal Z'(f))^{-1}.
\]
Therefore by Lemma \ref{lem2.2}, $(h,F)\in GL_{s(X)}(\Gamma(\mathcal Z'(f),\mathcal O_{X_f}/\mathcal J_f))$.

Next, considering $\phi(f)$ as a column-matrix, we extend it to the $s(X)\times s(X)$ matrix $(\phi(f),\tilde F)$.  Since $\mathfrak r^*(h)=
\mathfrak r^*(\phi(f))$,  the matrix $( \phi(f),\tilde F)\in GL_{s(X)}(\Gamma(\mathcal Z'(f),\mathcal O_{X_f}/\mathcal J_f))$ and 
 \[
 (h,\tilde F)\cdot (\phi(f),\tilde F)^{-1}=:G
\]
is an invertible matrix of the form $I_{s(X)}+L$, where entries of  $L$ belong to $\mathcal R$. Due to Lemma \ref{lem2.3}, $G$ is a product of a matrix $\exp(u)\cdot I_{s(X)}$ for some $u\in \mathcal R$  and finitely many  elementary matrices with entries in $\Gamma(\mathcal Z'(f),\mathcal O_{X_f}/\mathcal J_f)$. Due to \eqref{eq3.4}, each matrix in the factorization  of $G$ is the image of an invertible matrix, a multiple of $I_{s(X)}$ or an elementary matrix, with entries in $\mathcal O(X_f)$. Hence, there exists some $\tilde G\in GL_{s(X)}(\mathcal O(X_f))$ such that $\Phi(\tilde G)=G$, where $\Phi$ is  defined by the application of $\phi$ to matrix columns. Then we have $(h,\tilde F)=\Phi(\tilde G)\cdot (\phi(f),\tilde F)$.  In turn, this gives the identity of the first columns
\[
h=\Phi(\tilde G)\cdot\phi(f)=\phi(\tilde G\cdot f).
\]
Since $\tilde h:=\tilde G\cdot f\in U_{s(X)}(\mathcal O(X_f))$, the latter proves  the surjectivity of $\phi$, cf. (*).

The proof of the proposition is complete.
 \end{proof}
 
 As it was mentioned above,  Lemmas \ref{lem2.6} and \ref{lem1} and Proposition \ref{prop1.3} prove the Claim, i.e.,  ${\rm sr}(\mathcal O(X))\le s(X)$.\smallskip
 
Now, to complete the proof of Theorem \ref{teo1} let us show that 
\begin{equation}\label{eq3.7}
s(X)\le {\rm sr}(\mathcal O(X)).
\end{equation}
In order to prove the inequality it suffices to point out an irreducible element of $U_{s(X)}(\mathcal O(X))$.

Let $\tilde X\subset X$ be an irreducible component of the maximal dimension $d\, (:={\rm dim}\, X)$ and let
$x\in \tilde X$ be a regular point. Due to Cartan's Theorems there exist  holomorphic functions $f_1,\dots, f_d$ on $X$ and an open neighbourhood $U$ of $x$ such that the holomorphic map $F=(f_1,\dots, f_d): X\rightarrow\mathbb C^d$ is one-to-one on ${\rm cl}(U)$, maps $x$ to $0$ and $U$ biholomorphically onto the open unit ball $\mathbb B^d\subset\mathbb C^d$. Let $A(U)$ be the  algebra of holomorphic functions on $U$ which extend continuously to its boundary.  By our construction, $A(U)$ is isomorphic by means of the pullback of   
 $F|_{{\rm cl}(U)}$ to the (similarly defined) algebra $A(\mathbb B^d)$.  Consider the element 
\[
u=(z_1,z_3,\dots, z_{2s(X)-3}, p(z))\in U_{s(X)}(\mathcal O(\mathbb C^d)),\quad z=(z_1,\dots, z_d)\in\mathbb C^d,
\] 
where $p(z)=z_1z_2+z_3z_4+\cdots +z_{2s(X)-3\,}z_{2s(X)-2}-1$.

It was shown in the proof of \cite[Th.\,3.12]{CS2} that $u|_{{\rm cl}(\mathbb B^d)}\in U_{s(X)}(A(\mathbb B^d))$ is irreducible.  This implies that the restriction to ${\rm cl}(U)$ of $F^*u:=(f_1,f_3,\dots, f_{2s(X)-3}, p(F))\in U_{s(X)}(\mathcal O(X))$ is irreducible (in $U_{s(X)}(A(U))$) and therefore $F^*u$ is irreducible. 
This proves inequality \eqref{eq3.7}.

The proof of Theorem \ref{teo1} is complete.
\end{proof}

\sect{Proof of Proposition \ref{prop2.1}}
(1) Let $H\in SL_n(\Gamma(X,\mathcal O_X))$ be such that $\widehat{H}_n\in E_n(\mathcal O(X))$.  Since, due to Cartan's Theorem B and Lemma \ref{lem2.2}, the homomorphism $\widehat{\,}_n$ is surjective, the latter implies that $H=H^1\cdots H^l\, F$, where $H^i$, $1\le i\le l$, are alternating upper and lower triangular unipotent matrices in $SL_n(\Gamma(X,\mathcal O_X))$ and $F=I_n+G$ for some matrix $G$ with entries in the Jacobson radical of $\Gamma(X,\mathcal O_X))$, see Lemma \ref{lem2.1}. By virtue of Lemma \ref{lem2.3}, $F\in E_n(\Gamma(X,\mathcal O_X))$. Hence, $H\in E_n(\Gamma(X,\mathcal O_X))$ as well. 

The converse statement asserting that if $H\in E_n(\Gamma(X,\mathcal O_X))$ then $\widehat{H}_n\in E_n(\mathcal O(X))$ is obvious.

Next, let $ut_n(\Gamma(X,\mathcal O_X))$ be the minimal number $t$ such that every matrix in $E_n(\Gamma(X,\mathcal O_X))$ is a product of $t$ matrices such that each of them is either upper triangular with $1$ along the main diagonal or lower triangular with $1$ along the diagonal and the first matrix is upper triangular. Then we have for all $n\ge 2$, see \cite[Lm.\,7]{DV},
\begin{equation}\label{eq5.1}
t_n(\Gamma(X,\mathcal O_X))\le ut_n(\Gamma(X,\mathcal O_X))\le ut_2(\Gamma(X,\mathcal O_X))\le t_2(\Gamma(X,\mathcal O_X))+1.
\end{equation}
Also, due to \cite[Th.\,3.1]{IK2},
\begin{equation}\label{eq5.2}
t_2(\mathcal O(X))\le t_2(C(X))+2\le v(2d)+2.
\end{equation}
Once again, Cartan's Theorem B and Lemma \ref{lem2.2} imply that the homomorphism 
\[
\widehat{\,}_2: SL_2(\Gamma(X,\mathcal O_X))\rightarrow SL_2(\mathcal O(X))
\]
is surjective and its kernel $K_2$ consists of matrices with determinant one of the form $F=I_2+G$, where all entries of $G$ belong to the Jacobson radical of $\Gamma(X,\mathcal O_X)$. Applying the Gauss-Jordan elimination process and then the Whitehead lemma to such $F$ we write it as a product $F=F_1F_2 F_3F_4$, where $F_i\in K_2$, $1\le i\le 4$, are alternating upper and lower triangular unipotent matrices.
Now, the surjectivity of $\,\widehat{\,}_2$ and the normality of $K_2$ imply, in view of \eqref{eq5.2}, that each  $H\in E_2(\Gamma(X,\mathcal O_X))$  has a form 
\[
H=H^1 F H^2 \cdots H^l,
\]
where $H^i$, $1\le i\le l$, are alternating upper and lower triangular unipotent matrices in $SL_2(\Gamma(X,\mathcal O_X))$, $F\in K_2$ and $ l\le v(2d)+2$. Writing $F=F_1F_2 F_3F_4$ with $\{F_i\}\subset K_2$ as above such that the matrices $H_1$ and $F_1$ are both either upper or lower triangular we get that the matrices $F_4$ and $H^2$ are both either upper or lower triangular as well. Hence, we represent $H$ as a product $(H^1F_1)F_2 F_3 (F_4H_2)H_3\cdots H_l$ of at most $v(2d)+4$ triangular unipotent matrices in $SL_2(\Gamma(X,\mathcal O_X))$. Together with \eqref{eq5.1} this gives the required inequality 
\[
\sup_n t_n(\Gamma(X,\mathcal O_X))\le v(2d)+5.
\]

(2) Due to  \cite[Th.\,2.3]{IK1}, \cite[Th.\,20(b)]{DV}, our Theorem \ref{teo1} and inequalities \eqref{eq5.1}, 
\eqref{eq5.2} $t_n(\Gamma(X,\mathcal O_X))\le 6$
for all  $n\ge  n(d):=(\lfloor \frac d2 \rfloor +2)\cdot (\lfloor \frac{v(2d)}{2}\rfloor+3)$.\smallskip

(3) Since by the result of \cite{H} the space $X$ is homotopic to a $CW$ complex of dimension $d$ and the homotopy groups $\pi_i(SL_n(\mathbb C))$, $i=1,2$, are trivial, for $d=1,2$ each map in $C(X, SL_n(\mathbb C))$ is null-homotopic, see, e.g., \cite[p.\,351,\,(7.4)]{Hu}. Thus in these cases $SL_n(\Gamma(X,\mathcal O_X))=E_n(\Gamma(X,\mathcal O_X))$ for all $n$ by the first statement of part (1) of the proposition. 

Next, if $d=1$, then, by Theorem \ref{teo1}, ${\rm sr}(\Gamma(X,\mathcal O_X))=1$ so that $t_n(\Gamma(X,\mathcal O_X))\le 4$  for all $n\ge 2$ by \cite[Lm.\,9]{DV}. One easily shows that
$4$ in the previous inequalities is optimal because any nontrivial diagonal matrix in $SL_n(\Gamma(X,\mathcal O_X))$, $n\ge 2$, cannot be written as a product of less than four triangular unipotent matrices.

If $d=2$, then using \cite[Th.\,5.2]{IK2} and arguments similar to those of the proof of part (1) above we obtain that $\sup_n t_n(\Gamma(X,\mathcal O_X))\le ut_2(\Gamma(X,\mathcal O_X))\le ut_2(\mathcal O(X))+2\le 7$. We leave the details to the readers.

The proof of the proposition is complete.\medskip

\noindent {\bf Acknowledgment.} I thank the anonymous referee for useful comments.


\begin{thebibliography}{}

\bibitem[AM]{AM}
M. F. Atiyah, I. G. Macdonald, Introduction to commutative algebra, Addison-Wesley Publishing Co., 1969.

\bibitem[B]{B}
H. Bass, $K$-theory and stable algebra, Publ. Mat. 
I.H.E.S. {\bf 22} (1964), 5--60.


\bibitem[C]{C}
P. M. Cohn, On the structure of the $GL_2$ of a ring, Inst. Hautes \'{E}tudes Sci. Publ. Math. {\bf 30} (1966), 5--53.

\bibitem[CL]{CL}
G. Corach and A. R. Larotonda, Stable range in Banach algebras, J. Pure Appl. Algebra {\bf 32} (1984),  289--300.

\bibitem[CS1]{CS1}
G. Corach and F. Su\'{a}rez, Stable rank in holomorphic function algebras, 
Illinois J. Math. {\bf 29} (1985), 627--639.

\bibitem[CS2]{CS2}
G. Corach and F. Su\'{a}rez, Dense morphisms in commutative Banach algebras, Trans. Amer. Math. Soc. {\bf 304} (2) (1987), 537--547.

\bibitem[DV]{DV}
R. K. Dennis and L. N. Vaserstein, On a question of M. Newman on the number of commutators, J. Algebra {\bf 118} (1988), 150--161.

\bibitem[F]{F}
F. Forstneri\v{c}, The Oka principle for sections of stratified fiber bundles, Pure Appl. Math. Q. {\bf 6} (2010), 843--874.
 
\bibitem[G]{G}
M. Gromov, Oka's principle for holomorphic sections of elliptic bundles, J. Amer. Math. Soc. {\bf 2} (1989), 851--897.


\bibitem[GR]{GR}
H. Grauert and R. Remmert, Theory of Stein spaces,  Springer, Berlin, 1979.

\bibitem[H]{H}
H. A. Hamm, Zum Homotopietyp Steinscher R\"{a}ume, J. reine angew. Math. {\bf 338} (1983), 121--135.


\bibitem[Hu]{Hu}
S. T. Hu, Mappings of a normal space into an absolute neighbourhood retract, Trans. Amer. Math. Soc. {\bf 64} (1948), 336--358.

\bibitem[Hus]{Hus}
D. Husemoller, Fibre bundles, Springer-Verlag, New York, 1994.

\bibitem[IK1]{IK1}
B. Ivarsson and F. Kutzschebauch, Holomorphic factorization of mappings into $SL_n(\Co)$, Ann.  Math. {\bf 175} (2012), 45--69.

\bibitem[IK2]{IK2}
B. Ivarsson and F. Kutzschebauch, 
On the number of factors in the unipotent factorization of holomorphic mappings into $SL_2(\mathbb C)$, Proc. Amer. Math. Soc. {\bf 140}  (2012), no. 3, 823--838.





\bibitem[R]{R}
K. J. Ramspott, Stetige und holomorphe Schnitte in B\"{u}ndeln mit homogener Faser, Math. Z. {\bf 89} (1965), 234--246.

\bibitem[Ru]{Ru}
 L. A. Rubel, Solution of problem 6117, Amer. Math. Monthly {\bf 85} (1978), 505--506.

\bibitem[S]{S}
A. A. Suslin, On the structure of the special linear group over rings of polynomials, Izv. Akad. Nauk SSSR Ser. Mat. {\bf 41} (2) (1977), 235--252.




\bibitem[V1]{V1}
L. N. Vaserstein, The stable range for rings and the dimension of topological spaces, Functional
Anal. Appl., vol. 5 (1971), 102--110.

\bibitem[V3]{V3}
L. N. Vaserstein, Reduction of a matrix depending on parameters to a diagonal form by addition operations, Proc. Amer. Math. Soc. {\bf 103} (1988), 741--746.

\bibitem[W]{W}
W. van der Kallen,  $SL_3(\mathbb C[X])$ does not have bounded word length. In: Algebraic $K$-theory, Part I (Oberwolfach, 1980), Lecture Notes in Math. {\bf 966}, Springer (1982), 357--361.

\end{thebibliography}
\end{document}